\documentclass[12pt,leqno]{article}
\usepackage[doc]{optional}
\usepackage{color}
\usepackage{float}
\usepackage{soul}
\usepackage{graphicx}
\definecolor{labelkey}{rgb}{0,0.08,0.45}
\definecolor{refkey}{rgb}{0,0.6,0.0}
\definecolor{Brown}{rgb}{0.45,0.0,0.05}
\definecolor{lime}{rgb}{0.00,0.8,0.0}
\definecolor{lblue}{rgb}{0.5,0.5,0.99}

\usepackage{multirow}
\usepackage{mathpazo}

\usepackage{amsmath}

\usepackage{stmaryrd}
\usepackage{amssymb}
\usepackage{theorem}

\usepackage{hyperref}

\newcommand{\sepp}{\setlength{\itemsep}{-2pt}}

\newcommand{\weakly}{\ensuremath{\:{\rightharpoonup}\:}}

\newcommand{\nnn}{\ensuremath{{n\in{\mathbb N}}}}

\newcommand{\menge}[2]{\big\{{#1}~\big |~{#2}\big\}}

\newcommand{\fenv}[1]%
{\ensuremath{\,\overrightarrow{\operatorname{env}}_{#1}}}
\newcommand{\benv}[1]%
{\ensuremath{\,\overleftarrow{\operatorname{env}}_{#1}}}

\newcommand{\scal}[2]{\left\langle{#1},{#2}  \right\rangle}

\newcommand{\exi}{\ensuremath{\exists\,}}
\newcommand{\zeroun}{\ensuremath{\left]0,1\right[}}
\newcommand{\RR}{\ensuremath{\mathbb R}}
\newcommand{\ball}{\ensuremath{\operatorname{ball}}}

\newcommand{\NN}{\ensuremath{\mathbb N}}

\newcommand{\reli}{\ensuremath{\operatorname{ri}}}
\newcommand{\inte}{\ensuremath{\operatorname{int}}}
\newcommand{\sri}{\ensuremath{\operatorname{sri}}}

\newcommand{\ran}{\ensuremath{\operatorname{ran}}}

\newcommand{\aff}{\ensuremath{\operatorname{aff}}}

\newcommand{\lspan}{\ensuremath{{\operatorname{span}}\,}}

\newcommand{\cran}{\ensuremath{\overline{\operatorname{ran}}\,}}
\newcommand{\Fix}{\ensuremath{\operatorname{Fix}}}
\newcommand{\Id}{\ensuremath{\operatorname{Id}}}

{\begin{list}{}{%
\settowidth{\labelwidth}{\textrm{#1~}}%
\setlength{\leftmargin}{\labelwidth+\labelsep}}}
{\end{list}}
\newtheorem{theorem}{Theorem}[section]
\newtheorem{lemma}[theorem]{Lemma}
\newtheorem{corollary}[theorem]{Corollary}
\newtheorem{proposition}[theorem]{Proposition}
\newtheorem{definition}[theorem]{Definition}
\theoremstyle{plain}{\theorembodyfont{\rmfamily}
}
\theoremstyle{plain}{\theorembodyfont{\rmfamily}
}
\theoremstyle{plain}{\theorembodyfont{\rmfamily}
}
\theoremstyle{plain}{\theorembodyfont{\rmfamily}
\newtheorem{example}[theorem]{Example}}
\newtheorem{fact}[theorem]{Fact}
\theoremstyle{plain}{\theorembodyfont{\rmfamily}
\newtheorem{remark}[theorem]{Remark}}

\allowdisplaybreaks 

\begin{document}

\title{{Linear and strong convergence of algorithms\\ 
involving averaged nonexpansive operators}}

\author{
Heinz H.\ Bauschke\thanks{
Mathematics, University
of British Columbia,
Kelowna, B.C.\ V1V~1V7, Canada. E-mail:
\texttt{heinz.bauschke@ubc.ca}.},~
Dominikus Noll\thanks{
Universit{\'e} Paul Sabatier,
Institut de Math{\'e}matiques,
118 route de Narbonne, 31062 Toulouse, France.
E-mail: \texttt{noll@mip.ups-tlse.fr}.}~~~and~
Hung M.\ Phan\thanks{
Mathematics, University
of British Columbia,
Kelowna, B.C.\ V1V~1V7, Canada. E-mail:
\texttt{hung.phan@ubc.ca}.}
}

\date{February 21, 2014}

\maketitle

%
\begin{abstract} \noindent
We introduce regularity notions for averaged nonexpansive operators.
Combined with regularity notions of their fixed point sets,
we obtain linear and strong convergence results for quasicyclic,
cyclic, and random iterations.
New convergence results on the Borwein--Tam
method (BTM) and on the 
cylically anchored Douglas--Rachford algorithm (CADRA) are also presented. 
Finally, we provide a numerical comparison
of BTM, CADRA and the classical method of cyclic projections for
solving convex feasibility problems.
\end{abstract}

{\small
\noindent
{\bfseries 2010 Mathematics Subject Classification:}
{Primary 65K05; Secondary 47H09, 90C25.
}

\noindent {\bfseries Keywords:}
Averaged nonexpansive mapping, 
Borwein--Tam method, 
bounded linear regularity, 
convex feasibility problem, 
convex set,
cyclically anchored Douglas--Rachford algorithm,
Douglas--Rachford algorithm, 
nonexpansive operator, 
projection, 
random method,
transversality. 
}

\section{Overview}

Throughout this paper, $X$ is a real Hilbert space with
inner product $\scal{\cdot}{\cdot}$ and induced norm $\|\cdot\|$.
The convex feasibility problem asks to find a point in the
intersection of convex sets. This is an important problem in
mathematics and engineering; see, e.g.,
\cite{BB}, \cite{BC2011}, \cite{Ceg}, 
\cite{CZ}, \cite{Comb96}, \cite{GK}, \cite{GR}, \cite{Yamada},
and the references therein.

Oftentimes, the convex sets are given as fixed point sets of
projections or (more generally) averaged nonexpansive operators. 
In this case, weak convergence to a solution is guaranteed but
the question arises under which circumstances can we guarantee
strong or even linear convergence.
The situation is quite clear for 
projection algorithms; see, e.g., \cite{BB} and also \cite{KL}. 

\emph{
The aim of this paper is to provide verifiable sufficient
conditions for strong and linear convergence of algorithms based
on iterating convex combinations of averaged nonexpansive
operators.
}

Our results can be nontechnically summarized as follows:
\emph{If each operator is well behaved and the fixed point sets
relate well to each other, then the algorithm converges
strongly or linearly.}

Specifically, we obtain the following main results
on iterations of averaged nonexpansive mappings:

\begin{itemize}
\item
If each operator is boundedly linearly regular and the family
of corresponding fixed point sets is boundedly linearly regular, then
quasicyclic averaged algorithms converge linearly (Theorem~\ref{t:main1}). 
\item
If each operator is boundedly regular and the family of
corresponding fixed point sets is boundedly regular,
then cyclic algorithms converge strongly (Theorem~\ref{t:main3}). 
\item 
If each operator is boundedly regular and the family
of corresponding fixed point sets is innately boundedly regular,
then random sequential algorithms converge strongly
(Theorem~\ref{t:main2}). 
\end{itemize}

We also focus in particular on algorithms featuring 
the Douglas--Rachford splitting operator and obtain new
convergence results on the Borwein--Tam method and
the cyclically anchored Douglas--Rachford algorithm.

The remainder of the paper is organized as follows.
In Sections~\ref{s:blrop} and \ref{s:avnon}, 
we discuss (boundedly) linearly regular and averaged nonexpansive
operators. 
The bounded linear regularity of the Douglas--Rachford operator
in the transversal case is obtained in Section~\ref{s:Hung}. 
In Section~\ref{s:Fejer}, we recall the key notions of Fej\'er
monotonticity and regularity of collections of sets. 
Our main convergence result on quasicyclic algorithms is
presented in Section~\ref{s:main}. 
In Section~\ref{s:random}, we turn to strong convergence results
for cyclic and random algorithms. 
Applications and numerical results are provided in
Section~\ref{s:apps}.
Notation in this paper is quite standard and follows mostly 
\cite{BC2011}. 

\section{Operators that are (boundedly) linearly regular}

Our linear convergence results depend crucially on the concepts
of (bounded) linear regularity which we introduce now.

\label{s:blrop}

\begin{definition}[(bounded) linear regularity]
\label{d:blrop}
Let $T\colon X\to X$ be such that $\Fix T\neq\varnothing$.
We say that:
\begin{enumerate}
\item 
$T$ is 
\textbf{linearly regular}
with constant $\kappa \geq 0$ if 
\begin{equation}
(\forall x\in X)\quad
d_{\Fix T}(x) \leq \kappa\|x-Tx\|.
\end{equation}
\item
$T$ is 
\textbf{boundedly linearly regular}
if 
\begin{equation}
(\forall \rho>0)(\exi\kappa \geq 0)(\forall x\in \ball(0;\rho))\quad
d_{\Fix T}(x) \leq \kappa\|x-Tx\|;
\end{equation}
note that in general $\kappa$ depends on $\rho$, which we
sometimes indicate by writing $\kappa = \kappa(\rho)$. 
\end{enumerate}
\end{definition}

We clearly have the implication
\begin{equation}
\label{e:0126a}
\text{
linearly regular $\Rightarrow$ boundedly linearly regular.}
\end{equation}

\begin{example}[relaxed projectors]
Let $C$ be a nonempty closed convex subset of $X$ and
let $\lambda \in \left]0,2\right]$.
Then 
$T = (1-\lambda)\Id+\lambda P_C$ is linearly regular with
constant $\lambda^{-1}$.
\end{example}
\begin{proof}
Indeed, $\Fix T = C$ and $(\forall x\in X)$ 
$d_C(x) = \|x-P_Cx\| = \lambda^{-1}\|x-Tx\|$. 
\end{proof}

The following example shows that an operator may be boundedly
linearly regular yet not linearly regular.
This illustrates that the converse of the implication \eqref{e:0126a} 
fails.

\begin{example}[thresholder]
Suppose that $X=\RR$ and 
set 
\begin{equation}
Tx = \begin{cases}
0, &\text{if $|x|\leq 1$;}\\
x-1, &\text{if $x>1$;}\\
x+1, &\text{if $x<-1$.}
\end{cases}
\end{equation}
Then $T$ is boundedly linearly regular with $\kappa(\rho) =
\max\{\rho,1\}$;
however, $T$ is not linearly regular.
\end{example}
\begin{proof}
Let $x\in X$.
Since $\Fix T=\{0\}$, we deduce 
\begin{equation}
d_{\Fix T}(x) = |x| = \max\big\{|x|,1\big\}\min\big\{|x|,1\big\}
\end{equation}
and 
\begin{equation}
|x-Tx| = 
\begin{cases}
|x|, &\text{if $|x|\leq 1$;}\\
1, &\text{if $|x|>1$}
\end{cases}
= \min\big\{|x|,1\big\}. 
\end{equation}
If $x\notin\Fix T$, then $d_{\Fix T}(x)/|x-Tx| = \max\{|x|,1\}$
and the result follows.
\end{proof}

\begin{theorem}
\label{t:chess}
Let $T\colon X\to X$ be linear and nonexpansive with 
$\ran(\Id-T)$ closed.
Then $T$ is linearly regular. 
\end{theorem}
\begin{proof}
Set $A = \Id-T$.
Then $A$ is maximally monotone by \cite[Example~20.26]{BC2011},
and 
$(\Fix T)^\perp = (\ker A)^\perp = \cran A^* = \cran A =
\cran(\Id-T) = \ran(\Id-T)$ using \cite[Proposition~20.17]{BC2011}. 
By the Closed Graph Theorem
(see, e.g., \cite[Theorem~8.18]{Deutsch}), there exists 
$\beta>0$ such that
\begin{equation}
\big(\forall z\in \ker(A)^\perp)\big)
\quad
\|Az\|\geq \beta \|z\|.
\end{equation}
Now let $x\in X$ and split $x$ into 
$x=y+z$, where
$y=P_{\ker A}x = P_{\Fix T}x$
and 
$z= P_{(\ker A)^\perp}x = P_{\ran A}x = P_{\ran(\Id-T)}x$. 
Then
\begin{equation}
\|x-Tx\| = \|Ax\| = \|A(y+z)\| = \|Az\|
\geq\beta\|z\|=\beta\|x-P_{\Fix T}x\|=\beta d_{\Fix T}(x)
\end{equation}
and the result follows.
\end{proof}

\begin{example}[Douglas--Rachford operator for two subspaces]
\label{ex:DR1}
Let $U$ and $V$ be closed subspaces of $X$ such that
$U+V$ is closed, and set 
$T = P_VP_U + P_{V^\perp}P_{U^\perp}$.
Then 
$\Fix T = (U\cap V)+(U^\perp \cap V^\perp)$, and 
$\ran(\Id-T) = (U+V)\cap (U^\perp + V^\perp)$ is closed;
consequently, $T$ is linearly regular. 
\end{example}
\begin{proof}
The formula for $\Fix T$ is in, e.g., \cite{BBNPW}. 
On the one hand, it is well known (see, e.g.,
\cite[Corollary~15.35]{BC2011}) that
$U^\perp+V^\perp$ is closed as well. 
On the other hand,
\cite[Corollary~2.14]{BMW2} implies that
$\ran(\Id-T)= (U+V)\cap (U^\perp + V^\perp)$.
Altogether, $\ran(\Id-T)$ is closed. 
Finally, apply Theorem~\ref{t:chess}.
\end{proof}

\begin{example}
\label{ex:2lines}
Suppose that $X=\RR^2$, let $\theta\in\left]0,\pi/2\right]$, set
$U=\RR\cdot(1,0)$, $V=\RR\cdot(\cos\theta,\sin\theta)$, and 
$T = P_VP_U+P_{V^\perp}P_{U^\perp}$. 
Then $T$ is linearly regular with rate $1/\sin(\theta)$.
\end{example}
\begin{proof}
Let $x\in X$. 
A direct computation (or \cite[Section~5]{BBNPW}) yields
\begin{equation}
T = \cos(\theta) \begin{pmatrix} \cos(\theta) & -\sin(\theta)\\
\sin(\theta) & \cos(\theta),
\end{pmatrix}
\end{equation}
i.e., $T$ shrinks the vector by 
$\cos(\theta)\in\left[0,1\right[$ and rotates it by $\theta$. 
Hence $\Fix T = \{0\}$ and 
\begin{equation}
d_{\Fix T}(x) = \|x\|. 
\end{equation}
On the other hand, using $1-\cos^2(\theta)=\sin^2(\theta)$, we obtain
\begin{equation}
\Id - T = 
\sin(\theta)
\begin{pmatrix}
\sin(\theta) & \cos(\theta)\\
-\cos(\theta) & \sin(\theta)
\end{pmatrix}
\end{equation}
and hence
\begin{equation}
\|x-Tx\| = \sin(\theta)\|x\|. 
\end{equation}
Altogether, 
$d_{\Fix T}(x) = \|x\| = (1/\sin(\theta))\sin(\theta)\|x\|
=(1/\sin(\theta))\|x-Tx\|$. 
\end{proof}

We conclude this section by comparing our notion of bounded linear
regularity to metric regularity of set-valued operators. 

\begin{remark}
Suppose that $T$ is firmly nonexpansive and thus the resolvent of
a maximally monotone operator $A$. 
Suppose that $\bar{x}\in X$ is 
such that $0\in A\bar{x}$, i.e., $\bar{x}\in\Fix T$. 
Then \emph{metric subregularity} of $A$ at $\bar{x}$ means
that there exists $\delta>0$ and $\gamma>0$ such that
$x\in B(\bar{x};\delta)$ $\Rightarrow$ $d_{A^{-1}0}(x) \leq
\gamma d_{Ax}(0)$.
In terms of $T$, this is expressed as
$x\in \ball(\bar{x};\delta)$ $\Rightarrow$ $d_{\Fix T}(x) \leq
\gamma \inf\|x-T^{-1}x\|$. 
If $x=Ty  \in \ball(\bar{x};\delta)$, then 
the Minty parametrization yields
\begin{equation}
d_{\Fix T}(Ty) \leq \gamma \|y-Ty\|;
\end{equation}
moreover, $d_{\Fix T}(y)\leq (1+\gamma)\|y-Ty\|$.
This is related to bounded linear regularity of $T$.
The interested reader is referred to \cite{DonRock} for further
information on metric subregularity;
see also \cite{ADG} and \cite{Kru}.
\end{remark}

\section{Averaged nonexpansive operators}

\label{s:avnon}

We work mostly within the class of averaged nonexpansive
mappings which have proven to be a good compromise between
generality and usability.

\begin{definition}
The mapping $T\colon X\to X$ is \emph{averaged nonexpansive}
if there exists $\lambda\in\left[0,1\right[$ and $N\colon X\to X$
nonexpansive such that $T = (1-\lambda)\Id +\lambda N$.
\end{definition}

The class of averaged nonexpansive operators is closed under compositions and convex
combinations, and it includes all firmly nonexpansive mappings;
see, e.g., \cite{Comb04} for further information. 

\begin{example}
\label{ex:Banach2}
Let $T\colon X\to X$ be $\beta$-Lipschitz with $\beta\in\zeroun$.
Then $T$ is averaged.
\end{example}
\begin{proof}
Let $\varepsilon\in\left]0,(1-\beta)/2\right[\subset\zeroun$. 
Then $(\beta+\varepsilon)/(1-\varepsilon)\in\zeroun$. 
Now $(1-\varepsilon)^{-1}T$ is $(1-\varepsilon)^{-1}\beta$-Lipschitz
and $-\varepsilon(1-\varepsilon)^{-1}\Id$ is
$\varepsilon(1-\varepsilon)^{-1}$-Lipschitz, hence
\begin{equation}
N = (1-\varepsilon)^{-1}T - \varepsilon(1-\varepsilon)^{-1}\Id
\end{equation}
is nonexpansive. Set $\lambda= 1-\varepsilon\in\zeroun$. 
Then $(1-\lambda)\Id + \lambda N = \varepsilon\Id + (1-\varepsilon)N=T$
and $T$ is therefore averaged. 
\end{proof}

\begin{fact}
{\rm (See, e.g., \cite[Proposition~4.25(iii)]{BC2011}.)}
\label{f:aver}
Let $T\colon X\to X$ be averaged nonexpansive.
Then there exists $\sigma >0$ such that
\begin{equation}
(\forall x\in X)(\forall z\in\Fix T)\quad
\sigma\|x-Tx\|^2 \leq \|x-z\|^2-\|Tx-z\|^2.
\end{equation}
\end{fact}

The following two properties are crucial to our subsequent
analysis.

\begin{corollary}[$\sigma(T)$ notation]
\label{c:aver}
Let $T\colon X\to X$ be averaged nonexpansive.
Then there exists $\sigma = \sigma(T) >0$ such that
for every nonempty subset $C$ of $\Fix T$, we have 
\begin{equation}
(\forall x\in X)\quad
\sigma \|x-Tx\|^2 \leq d_C^2(x)-d_C^2(Tx). 
\end{equation}
\end{corollary}

\begin{corollary}
\label{c:0123}
Let $I$ be a finite ordered index set, let
$(T_i)_{i\in I}$ be family of averaged nonexpansive operators
with $\sigma_i = \sigma(T_i)$, and
let $(\omega_i)_{i\in I}$ be in $[0,1]$ such that
$\sum_{i\in I}\omega_i=1$. Set $I_+ = \menge{i\in
I}{\omega_i>0}$, and set $\sigma_+ = \min_{i\in I_+}\sigma_i$. 
Let $x\in X$, and set $y = \sum_{i\in I}\omega_i T_ix$
Then 
\begin{subequations}
\begin{align}
\big(\forall z\in \bigcap_{i\in I_+}\Fix T_i\big)\quad
\|x-z\|^2 &\geq \|y-z\|^2 +
\sum_{i\in I}\omega_i\sigma_i\|x-T_ix\|^2\\
&\geq \|y-z\|^2 +\sigma_+\|x-y\|^2.
\end{align}
\end{subequations}
\end{corollary}
\begin{proof}
Indeed, we have
\begin{subequations}
\begin{align}
\|y-z\|^2 &\leq \sum_{i\in I}\omega_i\|T_ix-z\|^2
\leq\sum_{i\in I}\omega_i\big(\|x-z\|^2 -\sigma_i\|x-T_ix\|^2\big)\\
&=\|x-z|^2-\sum_{i\in I}\omega_i\sigma_i\|x-T_ix\|^2
\leq\|x-z\|^2-\sigma_+\|x-y\|^2, 
\end{align}
\end{subequations}
as required.
\end{proof}

\begin{lemma}
\label{l:sinus}
Let $T\colon X\to X$ be averaged nonexpansive such that
\begin{multline}
\label{e:skating}
(\forall \rho>0)(\exi\theta<1)(\forall x\in \ball(0;\rho))(\exi y\in\Fix
T)\\
\scal{x-y}{Tx-y}\leq\theta\|x-y\|\|Tx-y\|.
\end{multline}
Then $T$ is boundedly linearly regular;
moreover, $T$ is linearly regular if $\theta$
does not depend on $\rho$.
\end{lemma}
\begin{proof}
We abbreviate $\sigma(T)$ by $\sigma$. 
Let $\rho>0$  and let $x\in\ball(0;\rho)$. 
Obtain $\theta$ and $y\in\Fix T$ as in \eqref{e:skating}.
Then
\begin{subequations}
\begin{align}
\|x-Tx\|^2 &=\|x-y\|^2 + \|y-Tx\|^2 + 2\scal{x-y}{y-Tx}\\
&\geq \|x-y\|^2+\|y-Tx\|^2-2\theta\|x-y\|\|Tx-y\|\\
&= (1-\theta)\big(\|x-y\|^2+\|y-Tx\|^2\big)
+\theta\big(\|x-y\|-\|y-Tx\|\big)^2\\
&\geq (1-\theta)\|x-y\|^2.
\end{align}
\end{subequations}
Hence
$(1-\theta)^{-1}\|x-Tx\|^2 \geq d_{\Fix T}^2(x)$. 
\end{proof}

The following example can be viewed as a generalization of
Example~\ref{ex:2lines}.

\begin{example}
Suppose that $S\colon X\to X$ is linear such that $S^*=-S$ and
$(\forall x\in X)$ $\|Sx\|=\|x\|$.
Let $\alpha\in\left]0,\pi/2\right]$, let
$\beta\in\left]-1,1\right[$, and set
$T = \beta(\cos(\alpha)\Id+\sin(\alpha)S)$.
Then $T$ is linearly regular. 
\end{example}
\begin{proof}
Set $R=\cos(\alpha)\Id+\sin(\alpha)S$. 
Then $T=\beta R$ and $(\forall x\in X)$ $\|Rx\|=\|Sx\|=\|x\|$; hence
$\|T\|=|\beta|<1$. 
By Example~\ref{ex:Banach2}, $T$ is averaged. 
Furthermore, $(\forall x\in X)$
$\scal{x}{Tx} = \beta\cos(\alpha)\|x\|^2
=\cos(\alpha)\|x\|\|\beta Rx\| = \cos(\alpha)\|x\|\|Tx\|$. 
The linear regularity of $T$ thus follows from
Lemma~\ref{l:sinus}. 
\end{proof}

We conclude this section with some key inequalities.

\begin{lemma}[key inequalities]
\label{l:key}
Let $T\colon X\to X$ be averaged firmly nonexpansive and
boundedly linearly regular, and let $\rho>0$.
Suppose that $C$ is a nonempty subset of $\Fix T$. 
Then there exist $\alpha\in\left[0,1\right[$,
$\beta\in\left]0,1\right]$, and $\gamma>0$ such that
for every $x\in \ball(0;\rho)$, we have
\begin{align}
&d_{\Fix T}(Tx)\leq \alpha d_{\Fix T}(x);\label{e:key1}\\
&\beta d_{\Fix T}^2(x)\leq \big( d_{\Fix T}(x)-d_{\Fix
T}(Tx)\big)^2\leq\|x-Tx\|^2;\label{e:key2}\\
& d_{C}^2(Tx)\leq d_{C}^2(x)-\gamma d^2_{\Fix T}(x).\label{e:key3}
\end{align}
If $T$ is linearly regular, then these constants do not depend on
$\rho$. 
\end{lemma}
\begin{proof}
Let us obtain the constants $\kappa = \kappa(\rho)\geq 0$ from
bounded linear regularity and 
$\sigma = \sigma(T)$ from the averaged nonexpansiveness. 
Abbreviate $Z=\Fix T$, and let $x\in \ball(0;\rho)$.
Then 
$d^2_{Z}(Tx)\leq d^2_Z(x) \leq \kappa^{2}\|x-Tx\|^2
\leq \sigma^{-1}\kappa^2(d^2_{Z}(x)-d^2_{Z}(Tx))$
by Corollary~\ref{c:aver}. Hence 
\eqref{e:key1} holds with 
\begin{equation}
\alpha = 
\sqrt{\frac{\sigma^{-1}\kappa^2}{1+\sigma^{-1}\kappa^2}} \in
\left[0,1\right[. 
\end{equation}
Note that  $\alpha$ depends only on $T$ when $T$ is in addition
linearly regular. 
Next, we set 
\begin{equation}
\beta = (1-\alpha)^2 \in\left]0,1\right],
\quad\text{and}\quad \gamma = \sigma \kappa^{-2}, 
\end{equation}
which again depend only on $T$ in the presence of linear
regularity. 
Then, by \eqref{e:key1}, 
$d_Z(x)-d_Z(Tx) \geq (1-\alpha)d_Z(x)$. 
Since $d_Z$ is nonexpansive, we deduce
\begin{equation}
\beta d_Z^2(x) \leq 
\big(d_{Z}(x)-d_{Z}(Tx)\big)^2
\leq \|x-Tx\|^2,
\end{equation}
i.e., \eqref{e:key2}. 
Finally, using Corollary~\ref{c:aver}, 
we conclude that
\begin{align}
d_C^2(Tx) &\leq d_C^2(x)-\sigma\|x-Tx\|^2
\leq d_C^2(x)- \sigma\kappa^{-2}d_Z^2(x), 
\end{align}
i.e., \eqref{e:key3} holds. 
\end{proof}

\section{The Douglas--Rachford Operator for Tranversal Sets}

\label{s:Hung}

In this section,
$X$ is finite-dimensional, 
$A$ and $B$ are nonempty closed convex subsets of $X$
with $A\cap B\neq\varnothing$.
Moreover, $L=\aff(A\cup B)$, $Y=L-L=\lspan(B-A)$,
denote the affine span of $A\cup B$ and the corresponding
parallel space, respectively. 
We also set
\begin{equation}
T = P_BR_A + \Id - P_A,
\end{equation}
i.e., $T$ is the Douglas--Rachford operator for $(A,B)$.
Note that $T(L)\subseteq L$. 
Our next two results are essentially contained in \cite{Hung},
where even nonconvex settings were considered.
In our present convex setting, the proofs become
much less technical. 

\begin{proposition}
\label{p:0208a}
The following hold:
\begin{enumerate}
\item
\label{p:0208a1}
$\Fix T = (A\cap B) + N_{A-B}(0)
= (A\cap B) + \big(Y\cap N_{A-B}(0)\big)+Y^\perp$.
\item
\label{p:0208a2}
$L \cap \Fix T = (A\cap B)+(Y\cap N_{A-B}(0))$.
\item
\label{p:0208a3}
If $\reli A \cap \reli B \neq \varnothing$, then 
$\Fix T = (A\cap B)+Y^\perp$  and
$L \cap \Fix T = A\cap B$. 
\item
\label{p:0208a3.5}
If $\reli A \cap \reli B \neq \varnothing$, then 
$P_{\Fix T} = \Id-P_L + P_{A\cap B}P_L$. 
\item
\label{p:0208a4}
$(\forall\nnn)$ 
$T^n=\Id-P_L + T^nP_L$.
\item
\label{p:0208a5}
$\Id-T = P_L-TP_L$. 
\item
\label{p:0208a6}
If $\reli A \cap \reli B \neq \varnothing$, then 
$d_{\Fix T} = d_{A\cap B}\circ P_L$. 
\end{enumerate}
\end{proposition}
\begin{proof}
\ref{p:0208a1}:
This follows from \cite[Corollary~3.9]{BCL} and \cite[Theorem~3.5]{BLPW1}. 
\ref{p:0208a2}:
Clear from \ref{p:0208a1}. 
\ref{p:0208a3}:
Use \ref{p:0208a1}, \ref{p:0208a2}, and 
\cite[Theorem~3.5 and Theorem~3.13]{BLPW1}. 
\ref{p:0208a3.5}:
Write $L = \ell+Y$, where $\ell\in Y^\perp$.
Then $P_L(A\cap B)=A\cap B = \ell+P_Y(A\cap B)$ and hence 
$\Fix T = P_Y(A\cap B) \oplus (\ell+Y^\perp)$. 
Now use \cite[Proposition~28.1(i) and Proposition~28.6]{BC2011}. 
\ref{p:0208a4}: 
By \cite[Lemma~3.3]{BLPW1}, $P_A = P_AP_L$ and $P_B=P_BP_L$. 
Moreover, $P_L$ is affine. 
This implies $R_A = R_AP_L + P_L-\Id$, $P_LR_A = R_AP_L$,
and $P_BR_A = P_BP_LR_A = P_BR_AP_L$. 
It follow that $T = \Id-P_L+TP_L = \Id-P_L+P_LTP_L$. 
The result follows then by induction. 
\ref{p:0208a5}: 
Clear from \ref{p:0208a4}. 
\ref{p:0208a6}: 
Clear from \ref{p:0208a3.5}. 
\end{proof}

\begin{lemma}
\label{l:sochi1}
Suppose $\reli A \cap \reli B\neq\varnothing$,
and let $c\in A\cap B$.
Then there exists $\delta>0$ and $\theta < 1$ such that
\begin{equation}
\label{e:sochi2}
\big(\forall x \in L \cap \ball(c;\delta)\big)\quad 
\scal{P_Ax-R_Ax}{P_BR_Ax-R_Ax}\leq \theta
d_A(x)d_B(R_Ax);
\end{equation}
consequently, 
\begin{equation}
\label{e:sochi3}
\big(\forall x \in L \cap \ball(c;\delta)\big)\quad 
\|x-Tx\|^2 \geq \frac{1-\theta}{5}
\max\big\{ d^2_A(x),d^2_B(x)\big\}. 
\end{equation}
\end{lemma}
\begin{proof}
Since $\reli A \cap \reli B\neq\varnothing$,
we deduce from \cite[Lemma~3.1 and Theorem~3.13]{BLPW1} that 
\begin{equation}
\label{e:sochi1}
N_A(c)\cap \big(-N_B(c)\big)\cap Y = \{0\}.
\end{equation}
Now suppose that \eqref{e:sochi2} fails. 
Noting that $P_A-R_A = \Id-P_A$, we obtain
a sequence $(x_n)_\nnn$ in $L$ converging to $c$ 
and a sequence $\theta_n\to 1^-$ such that 
for every $\nnn$, 
\begin{equation}
\scal{P_Ax_n-R_Ax_n}{P_BR_Ax_n-R_Ax_n} >  \theta_n
\|P_Ax_n-R_Ax_n\|\|P_BR_Ax_n-R_Ax_n\|.
\end{equation}
Hence
\begin{equation}
\scal{
\frac{x_n-P_Ax_n}{\|x_n-P_Ax_n\|}}{\frac{P_BR_Ax_n-R_Ax_n}{\|P_BR_Ax_n-R_Ax_n\|}}
\to 1^-.
\end{equation}
Set $u_n = (x_n-P_Ax_n)/\|x_n-P_Ax_n\|\in Y \cap N_A(P_Ax_n)$
and $v_n = (P_BR_Ax_n-R_Ax_n)/\|P_BR_Ax_n-R_Ax_n\|\in Y  \cap - 
N_B(P_BR_Ax_n)$.
After passing to subsequences if necessary we assume that
$u_n\to u$ and $v_n\to v$. 
Then $\scal{u}{v}=1$ and thus $v=u$. 
Since $x_n\to c$, we deduce that $P_Ax_n\to P_Ac=c$,
$R_Ax_n\to c$, and $P_BR_Ax_n\to c$. 
Thus, $u\in N_A(c)$ and $-u\in N_B(c)$. 
Altogether,
$u\in N_A(c)\cap (-N_B(c))\cap Y\smallsetminus\{0\}$,
which contradicts \eqref{e:sochi1}. 
We thus have proved \eqref{e:sochi2}. 

Now let $x\in \ball(c;\delta)\cap L$. 
Because 
$d_B$ is nonexpansive and
$R_A-\Id = 2(P_A-\Id)$, we deduce with
the Cauchy--Schwarz inequality that 
\begin{subequations}
\begin{align}
d_B^2(x) &\leq \big( \|x-R_Ax\|+d_B(R_Ax)\big)^2=\big(2 d_A(x) + d_B(R_Ax)\big)^2\\
&\leq 5\big( d_A^2(x) + d_B^2(R_Ax)\big).
\end{align}
\end{subequations}
Using \eqref{e:sochi2}, we have
\begin{subequations}
\begin{align}
\label{e:0129a}
\|x-Tx\|^2 &= \|P_Ax-P_BR_Ax\|^2\\
&=\|(P_Ax-R_Ax)+(R_Ax-P_BR_Ax)\|^2\\
&=\|P_Ax-R_Ax\|^2+\|R_Ax-P_BR_Ax\|^2 +
2\scal{P_Ax-R_Ax}{R_Ax-P_BR_Ax}\\
&\geq d_A^2(x)+ d_B^2(R_Ax)
-2\theta d_A(x)d_B(R_Ax)\\ 
&=(1-\theta)\big( d_A^2(x)+d_B^2(R_Ax)\big)
+\theta\big( d_A(x)- d_B(R_Ax) \big)^2\\
&\geq 
(1-\theta)\big( d_A^2(x)+d_B^2(R_Ax)\big)\\
&\geq 
\frac{1-\theta}{5}\max\big\{ d^2_A(x),d^2_B(x)\big\},
\end{align}
\end{subequations}
as claimed. 
\end{proof}

\begin{lemma}
\label{l:Hung}
Suppose that $\reli A \cap\reli B\neq\varnothing$.
Then
\begin{equation}
(\forall \rho>0)(\exi\kappa>0)(\forall x\in
L\cap\ball(0;\rho))\quad
\|x-Tx\|\geq \kappa d_{A\cap B}(x).
\end{equation}
\end{lemma}
\begin{proof}
We argue by contradiction and assume the conclusion fails.
Then there exists a bounded sequence $(x_n)_\nnn$ in $L$
and a sequence $\varepsilon_n\to 0^+$ such that
\begin{equation}
\label{e:sochi4}
(\forall\nnn)\quad
\|x_n-Tx_n\| < \varepsilon_n d_{A\cap B}(x_n)\to 0.
\end{equation}
In particular, $d_{A\cap B}(x_n)>0$ and $x_n-Tx_n\to 0$. 
After passing to subsequences if necessary,
we assume that $x_n\to \bar{x}$.
Then $\bar{x}\in L \cap \Fix T$.
By Proposition~\ref{p:0208a}\ref{p:0208a3}, $\bar{x}\in A\cap B$. 
Using 
Lemma~\ref{l:sochi1} and after passing to another subsequence if
necessary, we obtain $\theta<1$ such that
\begin{equation}
\label{e:sochi5}
(\forall\nnn)\quad
\|x_n-Tx_n\|^2 \geq \frac{1-\theta}{5}
\max\big\{ d^2_A(x_n),d^2_B(x_n)\big\}. 
\end{equation}
Next, bounded linear regularity of $(A,B)$ (see
Fact~\ref{f:regu}\ref{f:regu6} below) yields
$\mu>0$ such that $(\forall\nnn)$
$d_{A\cap B}(x_n)\leq\mu\max\{d_A(x_n),d_B(x_n)\}$. 
Combining this with \eqref{e:sochi4} and \eqref{e:sochi5} yields
\begin{align}
(\forall\nnn)\quad 
\varepsilon_n^2 d^2_{A\cap B}(x_n)&>\|x_n-Tx_n\|^2 \geq 
\frac{1-\theta}{5}
\max\big\{ d^2_A(x_n),d^2_B(x_n)\big\}\\
&\geq 
\frac{1-\theta}{5\mu^2} d^2_{A\cap B}(x_n).
\end{align}
This is absurd since $\varepsilon_n\to 0^+$.
\end{proof}

We are now ready for the main result of this section.

\begin{theorem}[Douglas--Rachford operator for two transversal
sets] 
\label{t:Hung}
Suppose that the pair $(A,B)$ is \emph{transversal}, i.e., 
$\reli A\cap\reli B\neq\varnothing$. 
Then $T$ is boundedly linearly regular.
\end{theorem}
\begin{proof}
Write $L=\ell+Y$, where $\ell\in Y^\perp$,
let $\rho > 0$, and set $\rho_L = \|\ell\|+\rho$.
Now obtain $\kappa$ as in Lemma~\ref{l:Hung}, applied to
$\rho_L$.
Let $x\in\ball(0;\rho)$.
Then $\|P_Lx\| = \|\ell+P_Yx\| \leq
\|\ell\|+\|P_Yx\|\leq \|\ell\|+\|x\|\leq \rho_L$. 
Hence $\|P_Lx-TP_Lx\|\geq \kappa d_{A\cap B}(P_Lx)$. 
On the other hand, 
$\|P_Lx-TP_Lx\| = \|x-Tx\|$
and $d_{A\cap B}(P_Lx) = d_{\Fix T}(x)$ by
Proposition~\ref{p:0208a}\ref{p:0208a5}\&\ref{p:0208a6}. 
Altogether, $\|x-Tx\|\geq \kappa d_{\Fix T}(x)$. 
\end{proof}

\begin{remark}
Lemma~\ref{l:sochi1}, which lies at the heart of this section,
is proved in much greater generality in the recent paper \cite{Hung}.
The novelty here is to deduce bounded linear regularity of the
Douglas--Rachford operator (see Theorem~\ref{t:Hung}) 
in order to make it a useful building
block to obtain other linear and strong convergence results. 
\end{remark}

\section{Fej\'er Monotonicity and Set Regularities}

\label{s:Fejer}

\subsection{Fej\'er monotone sequences and convergence for one operator}

Since all algorithms considered in this paper generate Fej\'er monotone
sequences, we review this key notion next. 

\begin{definition}[Fej\'er monotone sequence]
Let $C$ be a nonempty subset of $X$, and let $(x_n)_\nnn$ be a sequence in
$X$. Then $(x_n)_\nnn$ is \emph{Fej\'er monotone} with respect to $C$ if
\begin{equation}
(\forall c\in C)(\forall\nnn)\quad
\|x_{n+1}-c\|\leq \|x_n-c\|.
\end{equation}
\end{definition}
Clearly, every Fej\'er monotone sequence is bounded. 
Let us now review some results concerning norm and linear convergence of
Fej\'er monotone sequences. 

\begin{fact}
\label{f:linconv}
{\rm (See, e.g., \cite[Proposition~1.6]{BB}.)}
Let $(x_n)_\nnn$ be a sequence in $X$, let $\bar{x}\in X$,
and let $p\in\{1,2,\ldots\}$. 
Suppose that 
$(x_n)_\nnn$ is Fej\'er monotone with respect to $\{\bar{x}\}$,
and that 
$(x_{pn})_\nnn$ converges linearly to $\bar{x}$.
Then $(x_n)_\nnn$ itself converges linearly to $\bar{x}$. 
\end{fact}

\begin{fact}
\label{f:Fejer}
Let $(x_n)_\nnn$ be a sequence in $X$ that is Fej\'er monotone
with respect to a nonempty closed convex subset $C$ of $X$.
Then the following hold:
\begin{enumerate}
\item
\label{f:Fejer1}
If there exists $\alpha\in\left[0,1\right[$ such that 
$(\forall\nnn)$ $d_C(x_{n+1})\leq\alpha d_C(x_n)$, then
$(x_n)_\nnn$ converges linearly to some point $\bar{x}\in C$; in
fact, 
\begin{equation}
(\forall\nnn)\quad \|x_n-\bar{x}\|\leq 2\alpha^n d_C(x_0).
\end{equation}
\item
\label{f:Fejer2}
If $C$ is an affine subspace and all weak cluster points of $(x_n)_\nnn$
belong to $C$, then $x_n\weakly P_Cx_0$. 
\end{enumerate}
\end{fact}
\begin{proof}
\ref{f:Fejer1}: See, e.g., \cite[Theorem~5.12]{BC2011}.
\ref{f:Fejer2}: See, e.g., \cite[Proposition~5.9(ii)]{BC2011}. 
\end{proof}

\begin{corollary}
\label{c:MacG}
Let $T\colon X\to X$ be averaged firmly nonexpansive and boundedly linearly
regular, with $\Fix T\neq\varnothing$. 
Then for every $x_0\in X$, the sequence $(T^nx_0)_\nnn$ converges linearly
to some point $\bar{x}\in \Fix T$.
If $\Fix T$ is an affine subspace, then $\bar{x}=P_{\Fix T}x_0$. 
\end{corollary}
\begin{proof}
Let $x_0\in X$. The sequence $(T^nx_0)_\nnn$ is bounded because $\Fix
T\neq\varnothing$. By \eqref{e:key1} of Lemma~\ref{l:key}, there exists
$\alpha\in\left[0,1\right[$ such that 
$(\forall\nnn)$ $d_{\Fix T}(x_{n+1})\leq \alpha d_{\Fix T}(x_n)$. 
Hence Fact~\ref{f:Fejer}\ref{f:Fejer1} implies linear convergence of
$(T^nx_0)_\nnn$. The remainder of the theorem follows from
Fact~\ref{f:Fejer}\ref{f:Fejer2}. 
\end{proof}

Corollary~\ref{c:MacG} implies the following example,
which was analyzed in much greater detail in \cite{BBNPW}. 

\begin{example}[Douglas--Rachford operator for two subspaces]
Let $U$ and $V$ be closed subspaces such that $U+V$ is closed,
let $x_0\in X$, 
and set $T=P_VP_U + P_{V^\perp}P_{U^\perp}$. 
Then $(T^nx_0)_\nnn$ converges linearly to $P_{\Fix T}x_0$.
\end{example}
\begin{proof}
$T$ is averaged (even firmly nonexpansive), and linearly
regular by Example~\ref{ex:DR1}. 
Now apply Corollary~\ref{c:MacG}. 
\end{proof}

\begin{example}[Douglas--Rachford operator for transversal sets]
Suppose that $X$ is finite-dimensional, and let $U$ and $V$ be closed 
convex subsets of $X$ such that $\reli U \cap \reli
V\neq\varnothing$. 
Let $x_0\in X$, 
and set $T=P_VR_U+\Id-P_U$. 
Then $(T^nx_0)_\nnn$ converges linearly to some point
$\bar{x}\in\Fix T$ such that $P_U\bar{x}\in U\cap V$. 
\end{example}
\begin{proof}
Combine Theorem~\ref{t:Hung} with Corollary~\ref{c:MacG}. 
\end{proof}

\subsection{Regularities for families of sets}

We now recall the notion of a collection of
regular sets and key criteria. 
This will be crucial in the formulation of 
the linear convergence results. 

\begin{definition}[(bounded) (linear) regularity]
Let $(C_i)_{i\in I}$ be a finite family of closed convex subsets of $X$ with
$C=\bigcap_{i\in I}C_i\neq\varnothing$.
We say that:
\begin{enumerate}
\item
$(C_i)_{i\in I}$ is \emph{linearly regular}
if $(\exi \mu>0)$ $(\forall x\in X)$
$d_C(x)\leq \max_{i\in I} d_{C_i}(x)$.
\item
$(C_i)_{i\in I}$ is \emph{boundedly linearly regular}
if $(\forall\rho>0)$ $(\exi \mu>0)$ $(\forall x\in \ball(0;\rho))$
$d_C(x)\leq \max_{i\in I} d_{C_i}(x)$.
\item
$(C_i)_{i\in I}$ is \emph{regular}
if for every sequence $(x_n)_\nnn$ in $X$, we have
$\max_{i\in I} d_{C_i}(x_n)\to 0$
$\Rightarrow$
$d_C(x_n)\to 0$.
\item
$(C_i)_{i\in I}$ is \emph{boundedly regular}
if for every bounded sequence $(x_n)_\nnn$ in $X$, we have
$\max_{i\in I} d_{C_i}(x_n)\to 0$
$\Rightarrow$
$d_C(x_n)\to 0$.
\end{enumerate}
\end{definition}

\begin{fact}
\label{f:regu}
Suppose that $I=\{1,\ldots,m\}$, and
let $(C_i)_{i\in I}$ be a finite family of closed convex subsets of $X$ with
$C=\bigcap_{i\in I}C_i\neq\varnothing$.
Then the following hold:\
\begin{enumerate}
\item
\label{f:regu1}
Suppose each $C_i$ is a subspace.
Then $(C_i)_{i\in I}$ is 
regular in any of the four senses if and only if 
$\sum_{i\in I} C_i^\perp$ is closed. 
\item
\label{f:regu1+}
Suppose each $C_i$ is a cone. 
Then $(C_i\cap C^\ominus)_{i\in I}$ is 
regular in any of the four senses if and only if 
$\sum_{i\in I} (C_i\cap C^\ominus)^\ominus$ is closed. 
\item
\label{f:regu1++}
Suppose each $C_i$ is a cone and $C=\{0\}$. 
Then $(C_i)_{i\in I}$ is 
regular in any of the four senses if and only if 
$\sum_{i\in I} C_i^\ominus$ is closed. 
\item
\label{f:regu2}
If $C_m\cap \inte(C_1\cap\cdots \cap C_{m-1})\neq\varnothing$,
then $(C_i)_{i\in I}$ is boundedly linearly regular. 
\item
\label{f:regu3}
If $(C_1,C_2)$, $(C_1\cap C_2,C_3)$, 
\ldots, $(C_{1}\cap \cdots \cap C_{m-1},C_m)$ are 
(boundedly) linearly regular, then so is
$(C_i)_{i\in I}$. 
\item
\label{f:regu4}
If $0\in\sri(C_1-C_2)$, then $(C_1,C_2)$
is boundedly linearly regular. 
\item
\label{f:regu5}
If each $C_i$ is a polyhedron, then
$(C_i)_{i\in I}$ is linearly regular. 
\item
\label{f:regu6}
If $X$ is finite-dimensional,
$C_1,\ldots,C_k$ are polyhedra,
and $C_1\cap \cdots C_k\cap \reli(C_{k+1})\cap
\cdots\cap\reli(C_m)\neq\varnothing$, then
$(C_i)_{i\in I}$ is boundedly linearly regular. 
\item
\label{f:regu7}
If $X$ is finite-dimensional,
then $(C_i)_{i\in I}$ is boundedly regular. 
\end{enumerate}
\end{fact}
\begin{proof}
\ref{f:regu1}: 
\cite[Theorem~5.19]{BB}. 
\ref{f:regu1+}: 
\cite[Theorem~3.28]{DH}.
\ref{f:regu1++}: 
\cite[Corollary~3.30]{DH}.
\ref{f:regu2}: 
\cite[Corollary~5.13]{BB}.
\ref{f:regu3}: 
\cite[Theorem~5.11]{BB}.
\ref{f:regu4}: 
\cite[Corollary~4.5]{BB93}.
\ref{f:regu5}: 
\cite[Corollary~5.26]{BB}.
\ref{f:regu6}: 
\cite[Theorem~5.6.2]{Thesis}.
\ref{f:regu7}: 
\cite{BB93}.
\end{proof}

\begin{definition}[innate regularity]
Let $(C_i)_{i\in I}$ be a finite family of closed convex subsets of $X$ with
$C=\bigcap_{i\in I}C_i\neq\varnothing$.
We say that $(C_i)_{i\in I}$ is \emph{innately boundedly regular}
if $(C_j)_{j\in J}$ is boundedly regular for every nonempty
subset $J$ of $I$. Innate regularity and innate (bounded)
linear regularity are defined analogously.
\end{definition}

Fact~\ref{f:regu} allows to formulate a variety of conditions
sufficient for innate regularity.
Here, we collect only some that are quite useful.

\begin{corollary}
\label{c:regu}
Let $(C_i)_{i\in I}$ be a finite family of closed convex subsets of $X$ with
$C=\bigcap_{i\in I}C_i\neq\varnothing$.
Then the following hold:
\begin{enumerate}
\item 
\label{c:regu1}
If $X$ is finite-dimensional,
then  $(C_i)_{i\in I}$ is innately boundedly regular.
\item 
\label{c:regu2}
If $X$ is finite-dimensional and 
$\bigcap_{i\in I} \reli C_i\neq\varnothing$,
then  $(C_i)_{i\in I}$ is innately linearly regular. 
\item 
\label{c:regu3}
If each $C_i$ is a subspace 
and $\sum_{j\in J}C_j^\perp$ is closed for every nonempty subset
$J$ of $I$, then 
 $(C_i)_{i\in I}$ is innately linearly regular. 
\end{enumerate}
\end{corollary}
\begin{proof}
\ref{c:regu1}: Fact~\ref{f:regu}\ref{f:regu7}. 
\ref{c:regu2}: Fact~\ref{f:regu}\ref{f:regu6}. 
\ref{c:regu3}: Fact~\ref{f:regu}\ref{f:regu1}. 
\end{proof}

\section{Convergence Results for Quasi-Cyclic Algorithms}

\label{s:main}

Unless otherwise stated, we assume from now on that 
\begin{equation}
(T_i)_{i\in I}
\end{equation}
is a finite family of nonexpansive operators from $X$ to $X$ with 
common fixed point set
\begin{equation}
Z = \bigcap_{i\in I} Z_i\neq\varnothing,
\quad\text{where}\quad
(Z_i)_{i\in I} = (\Fix T_i)_{i\in I}. 
\end{equation}

We are now ready for our first main result. 

\begin{theorem}[quasi-cyclic algorithm]
\label{t:main1}
Suppose that each $T_i$ is boundedly linearly regular and 
averaged nonexpansive. 
Suppose furthermore that $(Z_i)_{i\in I}$ is boundedly
linearly regular. 
Let $(\omega_{i,n})_{(i,n)\in I\times \NN}$ be such that
$(\forall\nnn)$ $\sum_{i\in I}\omega_{i,n}=1$
and $(\forall i\in I)$ $\omega_{i,n}\in[0,1]$.
Set $(\forall\nnn)$ $I_n = \menge{i\in I}{\omega_{i,n}>0}$ and 
suppose that 
$\omega_+ = \inf_{\nnn}\inf_{i\in I_n}\omega_{i,n}>0$.
Suppose that there exists $p\in\{1,2,\ldots\}$ such that 
$(\forall\nnn)$
$I_{n} \cup I_{n+1}\cup \cdots \cup I_{n+p-1} = I$.
Let $x_0\in X$ and 
generate a sequence $(x_n)_\nnn$ in $X$ by
\begin{equation}
(\forall\nnn)\quad x_{n+1} = \sum_{i\in I}\omega_{i,n} T_{i}x_n.
\end{equation}
Then $(x_n)_\nnn$ converges linearly to some point in $Z$.
\end{theorem}
\begin{proof}
Set $\sigma_+ = \min_{i\in I}\sigma_i$, 
where $\sigma_i = \sigma(T_i)$. 
Let $i\in I$. 
By assumption,
\begin{equation}
(\forall k\in\NN)(\exi m_k\in\{kp,\ldots,(k+1)p-1\})\quad
i\in I_{m_k}.
\end{equation}
Let $z\in Z$. Then
\begin{equation}
d_{Z_i}(x_{kp}) \leq d_{Z_i}(x_{m_k}) + \|x_{kp}-x_{m_k}\|
\leq d_{Z_i}(x_{m_k}) + \sum_{n=kp}^{m_k-1}\|x_n-x_{n+1}\|.
\end{equation}
Hence, by using Cauchy--Schwarz,
\begin{equation}
d^2_{Z_i}(x_{kp}) \leq
(m_k+1-kp)\Big( d^2_{Z_i}(x_{m_k}) +
\sum_{n=kp}^{m_k-1}\|x_n-x_{n+1}\|^2\Big). 
\end{equation}
Get $\beta_j$ as in \eqref{e:key2} (with $T$ replaced by $T_j$)
and set $\beta_+ = \min_{j\in I}\beta_j>0$. 
In view of Corollary~\ref{c:0123},
it follows that
\begin{subequations}
\begin{align}
\|x_{kp}-z\|^2 - \|x_{(k+1)p}-z\|^2
&\geq \|x_{m_k}-z\|^2 - \|x_{m_k+1}-z\|^2\\
&\geq \omega_+\sigma_+\|x_{m_k}-T_ix_{m_k}\|^2\\
&\geq \omega_+\sigma_+\beta_+d^2_{Z_i}(x_{m_k}).
\end{align}
\end{subequations}
On the other hand,
by Corollary~\ref{c:0123},
\begin{equation}
(\forall \nnn)\quad
\|x_n-z\|^2 - \|x_{n+1}-z\|^2
\geq \sigma_{+}\|x_n-x_{n+1}\|^2.
\end{equation}
In particular, $(x_n)_\nnn$ is Fej\'er monotone with respect to
$Z$. 
Now we combine all of the above:
\begin{subequations}
\begin{align}
d^2_{Z_i}(x_{kp}) 
&\leq (m_k+1-kp)\big( d^2_{Z_i}(x_{m_k}) +
\sum_{n=kp}^{m_k-1}\|x_n-x_{n+1}\|^2\big)\\
&\leq \underbrace{p\big(\omega_+^{-1}\sigma_+^{-1}\beta_+^{-1}
+\sigma_+^{-1}\big)}_{=\lambda}
\Big(\|x_{kp}-z\|^2 - \|x_{(k+1)p}-z\|^2\Big). 
\end{align}
\end{subequations}
Applying this with $z=P_Zx_{kp}$ (and releasing $i$) yields
\begin{equation}
\max_{i\in I} d^2_{Z_i}(x_{kp}) \leq \lambda
\big( d^2_Z(x_{kp})-d^2_Z(x_{(k+1)p})\big).
\end{equation}
On the other hand, bounded linear regularity yields $\mu\geq 1$
such that $(\forall\nnn)$
$d_Z(x_n)\leq \mu\max_{i\in I}d_{Z_i}(x_n)$. 
Altogether,
\begin{equation}
d_Z^2(x_{kp}) \leq \lambda\mu^2\big(
d^2_Z(x_{kp})-d^2_Z(x_{(k+1)p})\big).
\end{equation}
By Fact~\ref{f:Fejer}\ref{f:Fejer1}, the sequence 
$(x_{kp})_{k\in\NN}$ converges linearly to some point $\bar{z}\in Z$. 
It now follows from Fact~\ref{f:linconv} that $(x_n)_\nnn$
converges linearly to $\bar{z}$.
\end{proof}

Theorem~\ref{t:main1} is quite flexible in the amount of control
a user has in generating sequences. We point out two very popular
instances next.

\begin{corollary}[cyclic algorithm]
\label{c:cyclic}
Suppose that $I=\{1,\ldots,m\}$, and that
each $T_i$ is boundedly linearly regular and 
averaged nonexpansive. 
Suppose furthermore that $(Z_i)_{i\in I}$ is boundedly
linearly regular. 
Let $x_0\in X$ and 
generate a sequence $(x_n)_\nnn$ in $X$ by
\begin{equation}
(\forall\nnn) \quad x_{n+1} = T_{m}\cdots T_2T_1x_n.
\end{equation}
Then $(x_n)_\nnn$ converges linearly to some point in $Z$.
\end{corollary}

\begin{corollary}[parallel algorithm]
\label{c:parallel}
Suppose that $I=\{1,\ldots,m\}$, and that
each $T_i$ is boundedly linearly regular and 
averaged nonexpansive. 
Suppose furthermore that $(Z_i)_{i\in I}$ is boundedly
linearly regular. 
Let $x_0\in X$ and 
generate a sequence $(x_n)_\nnn$ in $X$ by
\begin{equation}
(\forall\nnn) \quad x_{n+1} =
\frac{1}{m}\sum_{i\in I}T_ix_n.
\end{equation}
Then $(x_n)_\nnn$ converges linearly to some point in $Z$.
\end{corollary}

Some concrete and new results will be considered in
Section~\ref{s:apps}; there are already several known results
that can be deduced from this framework (see, e.g., \cite{BB} and
\cite{KL}). 

\begin{remark}
We mention here the related frameworks 
by Kiwiel and {\L}opuch \cite{KL} 
who bundled regularity of the fixed point sets together with
regularity of the operators to study accelerated generalizations
of projection methods. Theirs and our techniques
find their roots in \cite{BB}; see also \cite{Thesis}. 
We feel that the approach presented here is more convenient for
applications; indeed, one first checks that the operators are well
behaved --- the algorithms will be likewise if the fixed point
sets relate well to each other. 
\end{remark}

We end this section with the following probabilistic result
whose basic form is due to Leventhal \cite{Leventhal}. 
The proof presented here is somewhat simpler and 
the conclusion is stronger. 

\begin{corollary}[probabilistic algorithm]
Suppose that 
each $T_i$ is boundedly linearly regular and 
averaged nonexpansive. 
Suppose furthermore that $(Z_i)_{i\in I}$ is boundedly
linearly regular. 
Let $x_0\in X$ and 
generate a sequence $(x_n)_\nnn$ in $X$ by
\begin{equation}
(\forall\nnn) \quad x_{n+1} =
T_ix_n
\end{equation}
with probability $\pi_i>0$. 
Then $(x_n)_\nnn$ converges linearly almost surely to a solution
in the sense that there exists a constant $\theta<1$, depending only on
$\|x_0\|$, such that 
\begin{equation}
(\forall\nnn)\quad
\mathbf{E}\, d^2_Z(x_{n+1}) \leq \theta d^2_Z(x_n).
\end{equation}
\end{corollary}
\begin{proof}
Let $z\in Z$, and let $\nnn$. 
Then
$\|x_{n+1}\| = \|T_ix_n\|
\leq \|T_ix_n-z\|+\|z\|\leq \|x_n-z\|+\|z\| \leq
\|x_0-z\|+\|z\|$,
hence every instance of $(x_n)_\nnn$ satisfies
$\sup_\nnn\|x_n\|\leq \|x_0-z\|+\|z\| =\rho$. 
Hence, by \eqref{e:key3} of Lemma~\ref{l:key}, we obtain $\gamma_i$ 
such that
\begin{equation}
\gamma_i d^2_{Z_i}(x_n) \leq 
d_Z^2(x_n) - d^2_{Z}(T_ix_n).
\end{equation}
On the other hand, by bounded linear regularity of
$(Z_1,\ldots,Z_m)$, 
we get $\mu>0$ such that 
\begin{equation}
\label{e:0110a}
\mu d_Z^2(x_n)\leq \sum_{i}\pi_i\gamma_i d^2_{Z_i}(x_n).
\end{equation}
Combining and taking the expected value, we deduce 
\begin{equation}
\mu d_Z^2(x_n) \leq d^2_Z(x_n) - \mathbf{E}\,d^2_Z(x_{n+1}),
\end{equation}
and the result follows with $\theta = 1-\mu$.
\end{proof}

\section{Convergence Results for Cyclic and Random Algorithms}

\label{s:random}

In this section, we focus on strong convergence 
results for algorithms which utilize
the operators either cyclically or
in a more general, not necessarily quasicyclic,
fashion. Simple examples involving projectors show that
linear convergence results are not to be expected. Accordingly, 
the less restrictive notion of (bounded) regularity is introduced
--- it is sufficient for strong convergence.

We start our analysis with the following notion which can be seen
as a qualitative variant of (bounded) linear regularity.

\begin{definition}[(bounded) regularity]
Let $T\colon X\to X$ be such that $\Fix T\neq\varnothing$. 
We say that:
\begin{enumerate}
\item 
$T$ is 
\textbf{regular}
if for every sequence $(x_n)_\nnn$ in $X$, we have
\begin{equation}
x_n-Tx_n\to 0
\quad\Rightarrow\quad
d_{\Fix T}(x_n)\to 0.
\end{equation}
\item
$T$ is 
\textbf{boundedly regular}
if for every sequence $(x_n)_\nnn$ in $X$, we have 
\begin{equation}
\text{$(x_n)_\nnn$ bounded and } x_n-Tx_n\to 0
\quad\Rightarrow\quad
d_{\Fix T}(x_n)\to 0.
\end{equation}
\end{enumerate}
\end{definition}
Comparing with Definition~\ref{d:blrop}, we note that 
\begin{equation}
\text{
linear regularity
$\Rightarrow$
regularity
}
\end{equation}
and that 
\begin{equation}
\text{
bounded linear regularity
$\Rightarrow$
bounded regularity. 
}
\end{equation}

These notions are much less restrictive than their quantitative
linear counterparts: 

\begin{proposition}
Let $T\colon X\to X$ be continuous, suppose that 
$X$ is finite-dimensional\,\footnote{Or, more generally, that $\ran
T$ is boundedly compact.} and that $\Fix T\neq\varnothing$.
Then $T$ is boundedly regular. 
\end{proposition}

We now turn to ``property (S)'', a notion first considered
by Dye \emph{et al.} in \cite{DKLR}.

\begin{definition}[property (S)]
Let $T\colon X\to X$ be nonexpansive such that $\Fix
T\neq\varnothing$. Then 
$T$ has property (S) with respect to $z\in \Fix T$
if for every bounded sequence $(x_n)_\nnn$ such that 
$\|x_n-z\|-\|Tx_n-z\|\to 0$, we have $x_n-Tx_n\to 0$. 
\end{definition}

\begin{proposition}
\label{p:0124a}
Let $T\colon X\to X$ be averaged nonexpansive such that $\Fix
T\neq\varnothing$. 
Then $T$ has property (S) with respect to $\Fix T$.
\end{proposition}
\begin{proof}
Let $(x_n)_\nnn$ be a bounded sequence in $X$ such that
$\|x_n-z\|-\|Tx_n-z\|\to 0$, where $z\in \Fix T$.
Clearly, $(\|x_n-z\|+\|Tx_n-z\|)_\nnn$ is bounded since
$(x_n)_\nnn$ and $(Tx_n)_\nnn$ are. It follows that
$\|x_n-z\|^2-\|Tx_n-z\|^2\to 0$. 
By Fact~\ref{f:aver}, $x_n-Tx_n\to 0$. 
\end{proof}

\begin{definition}[projective]
Let $T\colon X\to X$ be nonexpansive such that $\Fix
T\neq\varnothing$, and let $z\in\Fix T$. Then 
$T$ is \emph{projective} with respect to $z\in\Fix T$
if for every bounded sequence $(x_n)_\nnn$ such that 
$\|x_n-z\|-\|Tx_n-z\|\to 0$, we have $d_{\Fix T}(x_n)\to 0$.
We say that $T$ is \emph{projective} if it is projective
with respect to all its fixed points. 
\end{definition}

Projectivity implies property~(S):

\begin{lemma}
\label{l:0210a}
Let $T\colon X\to X$ be nonexpansive and suppose that
$T$ is projective with respect to 
$z\in \Fix T$. Then $T$ has property~(S) with respect to $z$. 
\end{lemma}
\begin{proof}
Observe that
\begin{subequations}
\label{e:0210a}
\begin{align}
(\forall x\in X)\quad
\|x-Tx\|&\leq \|x-P_{\Fix T}x\|+\|P_{\Fix T}x-Tx\|\\
&\leq 2\|x-P_{\Fix T}x\|=2 d_{\Fix T}(x).
\end{align}
\end{subequations}
Now let $(x_n)_\nnn$ be a bounded sequence such that
$\|x_n-z\|-\|Tx_n-z\|\to 0$. 
Since $T$ is projective with respect to $z$, 
we have $d_{\Fix T}(x_n)\to 0$. 
By \eqref{e:0210a}, $x_n-Tx_n\to 0$.
\end{proof}

The importance of projectivity stems from the following 
observation.

\begin{fact}
\label{f:wow}
Let $T\colon X\to X$ be nonexpansive such that $T$ is projective
with respect to some fixed point of $T$. 
Then $(T^nx_0)_\nnn$ converges strongly to a fixed point for
every starting point $x_0\in X$.
\end{fact}
\begin{proof}
See \cite[Lemma~2.8.(iii)]{B95}.
\end{proof}

\begin{proposition}
\label{p:wow}
Let $I=\{1,\ldots,m\}$, and 
let $(T_i)_{i\in I}$ be nonexpansive mappings with fixed point
sets $(Z_i)_{i\in I}$. 
Set $Z = \bigcap_{i\in I} Z_i$ and suppose that 
there exists $z\in Z$ such that 
each $T_i$ is projective with respect to $z$ and that
$(Z_i)_{i\in I}$ is boundedly regular. 
Then $T = T_m\cdots T_2T_1$ is projective with respect to $z$ as
well. Consquently, for every $x_0\in X$,
$(T^nx_0)_\nnn$ converges strongly to some point in $Z$.
\end{proposition}
\begin{proof}
We have $(\forall i\in I)$ 
$(\forall x\in X \smallsetminus Z_i)$
$(\forall z\in Z_i)$
$\|T_ix-z\|<\|x-z\|$,
i.e., each $T_i$ is attracting. 
By \cite[Proposition~2.10]{BB}, $T$ is attracting
and $\Fix T = Z$.
Now suppose that $(x_n)_\nnn$ is a bounded sequence in $X$ 
such that $\|x_n-z\|-\|Tx_n-z\|\to 0$. 
Note that
\begin{equation}
0\leq \sum_{i=1}^m \|T_{i-1}\cdots T_1x_n-z\|-\|T_{i}T_{i-1}\cdots
T_1x_n-z\|
=\|x_n-z\|-\|Tx_n-z\|\to 0,
\end{equation}
that each sequence $(T_{i-1}\cdots T_1x_n)_\nnn$ is bounded, and
that 
\begin{equation}
(\forall i\in I)\quad
\|T_{i-1}\cdots T_1x_n-z\|-\|T_{i}T_{i-1}\cdots T_1x_n-z\|\to
0.
\end{equation}
This has two consequences. 
First, 
\begin{equation}
(\forall i\in I)\quad
T_{i-1}\cdots T_1x_n - T_{i}T_{i-1}\cdots T_1x_n\to 0
\end{equation}
by Lemma~\ref{l:0210a}.
Second, 
\begin{equation}
(\forall i\in I)\quad
d_{Z_i}(T_{i-1}\cdots T_1x_n)\to 0
\end{equation}
because $T_i$ is projective with respect to $z$.
Altogether,
$(\forall i\in I)$
$d_{Z_i}(x_n)\to 0$. 
Since $(Z_i)_{i\in I}$ is boundedly regular, it follows that
$d_Z(x_n)\to 0$. 
Hence $T$ is projective with respect to $z$ and the result
now follows from Fact~\ref{f:wow}. 
\end{proof}

Property~(S) in tandem with bounded regularity implies
projectivity, which turns out to be crucial for the results
on random algorithms. 

\begin{proposition}
\label{p:0124b}
Let $T\colon X\to X$ be nonexpansive such that 
$\Fix T\neq \varnothing$, and let $z\in \Fix T$.
Suppose that $T$ satisfies
property~(S) with respect to $z$, and that $T$ is boundedly
regular. Then $T$ is projective with respect to $z$.
\end{proposition}
\begin{proof}
Let $(x_n)_\nnn$ be bounded such that $\|x_n-z\|-\|Tx_n-z\|\to
0$.
By property~(S), $x_n-Tx_n\to 0$. 
By bounded regularity, $d_{\Fix T}(x_n)\to 0$, as required. 
\end{proof}

The next result is quite useful.

\begin{corollary}
\label{c:0124c}
Let $T\colon X\to X$ be averaged nonexpansive and boundedly
regular such that $\Fix T\neq\varnothing$.
Then $T$ is projective with respect to $\Fix T$.
\end{corollary}
\begin{proof}
Combine Proposition~\ref{p:0124a} and Proposition~\ref{p:0124b}.
\end{proof}

We now obtain a powerful
strong convergence result for cyclic algorithms.

\begin{theorem}[cyclic algorithm]
\label{t:main3}
Set $I=\{1,\ldots,m\}$, and 
let $(T_i)_{i\in I}$ be family of averaged nonexpansive mappings from $X$
to $X$ with fixed point sets $(Z_i)_{i\in i}$, respectively. 
Suppose that each $T_i$ is boundedly regular,
that $Z = \bigcap_{i\in I} Z_i\neq\varnothing$, and that
$(Z_i)_{i\in I}$ is boundedly regular.
Then 
for every $x_0\in X$,
the sequence $((T_m\cdots T_1)^nx_0)_\nnn$
converges strongly to some point in $Z$.
\end{theorem}
\begin{proof}
By Corollary~\ref{c:0124c},
each $T_i$ is projective with respect to every point in $Z$. 
The result thus follows from Proposition~\ref{p:wow}. 
\end{proof}

Let us now turn to random algorithms. 

\begin{definition}[random map] 
The map $r\colon \NN\to I$ is a \emph{random map} for $I$
if $(\forall i\in I)$ $r^{-1}(i)$ contains infinitely many
elements. 
\end{definition}

\begin{fact}
\label{f:oldtimes}
{\rm (See \cite[Theorem~3.3]{B95}.)}
Suppose that $(T_i)_{i\in I}$ are projective with respect to a common
fixed point, and that 
$(Z_i)_{i\in I}$ is innately boundedly regular. 
Let $x_0\in X$, let $r$ be a random map for $I$, and generate a sequence
$(x_n)_\nnn$ in $X$ by
\begin{equation}
(\forall\nnn)\quad x_{n+1} = T_{r(n)}x_n.
\end{equation}
Then $(x_n)_\nnn$ converges strongly to some point in $Z$. 
\end{fact}

We are ready for our last main result.

\begin{theorem}[random algorithm]
\label{t:main2}
Suppose that each $T_i$ is averaged nonexpansive and boundedly
regular, and that 
$(Z_i)_{i\in I}$ is innately boundedly regular. 
Let $x_0\in X$, let $r$ be a random map for $I$, and generate a sequence
$(x_n)_\nnn$ in $X$ by
\begin{equation}
(\forall\nnn)\quad x_{n+1} = T_{r(n)}x_n.
\end{equation}
Then $(x_n)_\nnn$ converges strongly to some point $\bar{z}\in Z$. 
If $Z$ is an affine subspace, then $\bar{z}=P_Zx_0$. 
\end{theorem}
\begin{proof}
By Corollary~\ref{c:0124c}, each $T_i$ is projective with respect to $Z_i$
and hence with respect to $Z$. 
Now apply Fact~\ref{f:oldtimes} and
Fact~\ref{f:Fejer}\ref{f:Fejer2}. 
\end{proof}

\section{Applications and Numerical Results}

\label{s:apps}

\subsection{The Borwein--Tam Method (BTM)}

\label{s:BT}

In this section, $I=\{1,\ldots,m\}$ and
$(U_i)_{i\in I}$ is a family of closed convex subsets of $X$ with
\begin{equation}
U = \bigcap_{i\in I}U_i\neq\varnothing.
\end{equation}
Now set $U_{m+1}=U_1$,
\begin{equation}
(\forall i\in I)\quad
T_{i} = T_{U_{i+1},U_{i}} = P_{U_{i+1}}R_{U_i}+\Id -
P_{U_i},\;\;
Z_i = \Fix T_i,\;\;
Z = \bigcap_{i\in I} Z_i
\end{equation}
and define the \emph{Borwein--Tam operator} by
\begin{equation}
T = T_{m}T_{m-1}T_{m-2}\cdots T_{2}T_{1}.
\end{equation}

The following result is due to Borwein and Tam (see
\cite[Theorem~3.1]{BT}):

\begin{fact}[Borwein--Tam method (BTM)]
Let $x_0\in X$ and generate the sequence $(x_n)_\nnn$ by
\begin{equation}
(\forall\nnn)\quad x_{n+1}=T^nx_0.
\end{equation}
Then $(x_n)_\nnn$ converges weakly to a point $\bar{x}\in Z$
such that $P_{U_1}\bar{x}=\cdots = P_{U_m}\bar{x}\in U$.
\end{fact}

The following new results now follow from our analysis.

\begin{corollary}[transversal sets]
Suppose that $X$ is finite-dimensional and that
$\bigcap_{i\in I}\reli U_i\neq\varnothing$. 
Then the convergence of the Borwein--Tam method is with a linear rate.
\end{corollary}
\begin{proof}
Combine Theorem~\ref{t:Hung} with Corollary~\ref{c:cyclic}. 
\end{proof}

\begin{corollary}[subspaces]
\label{c:BTsub}
Suppose that each $U_i$ is a subspace\footnote{A simple
translation argument yields a version for affine subspaces with
a nonempty intersection.} with $U_i + U_{i+1}$ is closed, 
and that $(Z_i)_{i\in I}$ is boundedly
linearly regular. 
Then the convergence of the Borwein--Tam method is with a linear rate.
\end{corollary}
\begin{proof}
Combine Example~\ref{ex:DR1} with Corollary~\ref{c:cyclic}. 
\end{proof}

Of course, using Theorem~\ref{t:main1}, we can formulate various
variants for a general quasicyclic variant. 
We conclude this section with a random version.

\begin{example}[subspaces --- random version]
Suppose the hypothesis of Corollary~\ref{c:BTsub} holds.
Assume in addition that $(Z_i)_{i\in I}$ 
is innately boundedly regular. Let $r$ be a random map for $I$,
let $x_0\in X$, and set $(\forall\nnn)$ $x_{n+1}=T_{r(n)}x_n$.
Then $(x_n)_\nnn$ converges strongly to $P_Zx_0$.
\end{example}
\begin{proof}
Combine Example~\ref{ex:DR1} with Theorem~\ref{t:main2}. 
\end{proof}

\subsection{The Cyclically Anchored Douglas--Rachford Algorithm
(CADRA)}

In this section, we assume that $I=\{1,\ldots,m\}$,
that $A$ is a closed convex subset of $X$, also referred to as the
\emph{anchor}, 
and that $(B_i)_{i\in I}$ is a family of closed convex subsets of $X$
such that 
\begin{equation}
C = A \cap \bigcap_{i\in I} B_i\neq \varnothing.
\end{equation}
We set
\begin{equation}
(\forall i\in I)\quad T_i = P_{B_i}R_{A} + \Id-P_{A}, 
\;\; Z_i = \Fix T_i;\;\
Z = \bigcap_{i\in I} Z_i. 
\end{equation}
The \emph{Cyclically Anchored Douglas--Rachford Algorithm
(CADRA)} with starting point $x_0\in X$ 
generates a sequence $(x_n)_\nnn$ by
iterating
\begin{equation}
(\forall\nnn)\quad
x_{n+1} = Tx_n,\;\;
\text{where}\;\;
T = T_m\cdots T_2T_1.
\end{equation}

Note that when $m=1$, then CADRA coincides with the classical
Douglas--Rachford algorithm\footnote{This is not the case for the
BTM considered in the previous subsection.}. 

Let us record a central convergence result concerning the CADRA.

\begin{theorem}[CADRA]
\label{t:cadra}
The sequence $(x_n)_\nnn$ generated by CADRA converges weakly to a point
$\bar{x}\in Z$ such that $P_A\bar{x}\in C$. 
Furthermore, the convergence is linear provided that 
one of the following holds:
\begin{enumerate}
\item
\label{t:cadra1}
$X$ is finite-dimensional and that $\reli(A)\cap \bigcap_{i\in I}
\reli(B_i)\neq\varnothing$. 
\item
\label{t:cadra2}
$A$ and each $B_i$ is a subspace with $A+B_i$ closed
and that $(Z_i)_{i\in I}$ is boundedly linearly regular.
\end{enumerate}
\end{theorem}
\begin{proof}
The weak convergence follows from e.g.\
\cite[Theorem~5.22]{BB}. 
\ref{t:cadra1}: 
Now combine Theorem~\ref{t:Hung} with Corollary~\ref{c:cyclic}. 
\ref{t:cadra2}: 
Combine Example~\ref{ex:DR1} with Corollary~\ref{c:cyclic}. 
\end{proof}

One may also obtain a random version of CADRA by using
Theorem~\ref{t:main2}.

\subsection{Numerical experiments}

We now work in $X=\RR^{100}$. We set
$A = \RR^{50}_+\times\{0\}\subset X$, 
and we let each $B_i$ be a hyperplane with
normal vector in $\RR^{100}_{++}$, where $1\leq i\leq m$ and
$1\leq m\leq 50$. 
Using the programming language \texttt{julia} \cite{Julia}, 
we generated these data randomly, where for each
$m\in\{1,\ldots,50\}$, the problem
\begin{equation}
\text{find $x\in A \cap \bigcap_{i\in\{1,\ldots,m\}} B_i$}
\end{equation}
has a solution in $\reli A$.
We then choose 
10 random starting points in $\RR_+^{100}$, each with Euclidean
norm equal to $100$. Altogether, we obtain 50 problems and 500 instances
for each of the algorithms
Cyclic Projections (CycP), 
BTM, and CADRA applied to the sets $A,B_1,\ldots,B_m$. 
If $(x_n)_\nnn$ is the main sequence generated by one of these
algorithms and $(z_n)_\nnn = (P_{A}x_n)_\nnn$, then 
we terminate at stage $n$ when
\begin{equation}
\max\big\{d_{B_1}(z_n),\ldots,d_{B_m}(z_n)\big\} \leq 10^{-3}.
\end{equation}
We divide the 50 problems into 5 groups, depending on the value
of $m$. In Table~\ref{t:1}, 
we record the median of the number of iterations
required for each algorithm to terminate, and we also list the percentage 
that each algorithm is the fastest among the three.
\begin{table}[!h]
\centering
\begin{tabular}{|r|r|r|r|r|r|r|}
\hline
\multirow{2}{*}{Range of $m$}
& \multicolumn{2}{|c|}{CycP}
& \multicolumn{2}{|c|}{BTM}
& \multicolumn{2}{|c|}{CADRA}\\
\cline{2-7}
& Iterations & Wins  & Iterations & Wins  & Iterations & Wins  \\
\hline
1--10 & 79.5  & 10 & 78.5   & 52 &      80.0 & 43\\
\hline
11--20 & 391.0   &  0	 &  384.0   &   0 & 179.5 & 100\\
\hline
21--30 & 932.0   &  0  & 942.5 &  2  &  370.5 &  98\\
\hline
31--40 & 1,645.0 &  8  & 1,690.5 &  6  &  959.5 &  86\\
\hline
41--50 & 4,749.0 &  35  & 4,482.5 &  35  & 5,151.5 &  30\\
\hline
\end{tabular}
\caption{Median of number of iterations and number of wins}
\label{t:1}
\end{table}
Finally, we observe that CADRA performs quite well compared to CycP and
BTM, especially when the range of parameters keep the problems
moderately underdetermined.

\section*{Acknowledgments}
HHB was partially supported by the Natural Sciences and
Engineering Research Council of 
Canada and by the Canada Research Chair
Program.
DN acknowledges hospitality of the University of British Columbia
in Kelowna and support by the Pacific Institute of the Mathematical Sciences during the preparation of this paper.
HMP was partially supported by an NSERC accelerator grant of HHB.



\begin{thebibliography}{999}

\sepp 

\bibitem{ADG}
F.J.\ Arag\'on Artacho, A.L.\ Dontchev, and M.H.\ Geoffroy,
Convergence of the proximal point method for metrically regular
mappings,
\emph{ESAIM Proceedings} 17 (2007), 1--8. 


\bibitem{B95}
H.H.\ Bauschke,
A norm convergence result on random products of relaxed projections
in Hilbert space,
\emph{Transactions of the AMS}~347(4) (April 1995), 1365--1373. 

\bibitem{Thesis}
H.H.\ Bauschke,
\emph{Projection Algorithms and Monotone Operators},
PhD thesis, Simon Fraser University,
Burnaby, B.C., Canada, 1996. 

\bibitem{BBNPW}
H.H.\ Bauschke, J.Y.\ Bello Cruz, T.T.A.\ Nghia,
H.M.\ Phan, and X.\ Wang,
The rate of linear conergence of the Douglas--Rachford algorithm
for subspaces is the cosine of the Friedrichs angle,
\texttt{http://arxiv.org/abs/1309.4709}. 

\bibitem{BB93}
H.H.\ Bauschke and J.M.\ Borwein,
On the convergence of von Neumann's alternating projection
algorithm,
\emph{Set-Valued Analysis}~1 (1993), 185--212. 

\bibitem{BB}
H.H.\ Bauschke and J.M.\ Borwein,
On projection algorithms for solving convex feasibility problems,
\emph{SIAM Review}~38(3) (1996), 367--426.

\bibitem{BC2011}
H.H.\ Bauschke and P.L.\ Combettes,
\emph{Convex Analysis and Monotone Operator Theory in Hilbert Spaces},
Springer, 2011.

\bibitem{BCL}
H.H.\ Bauschke, P.L.\ Combettes, and D.R.\ Luke, 
Finding best approximation pairs relative to two closed convex
sets in Hilbert spaces, 
\emph{Journal of Approximation Theory}~127 (2004), 178--192.

\bibitem{BLPW1}
H.H.\ Bauschke, D.R.\ Luke, H.M.\ Phan, and X.\ Wang,
Restricted normal cones and the method of alternating
projections: theory,
\emph{Set-Valued and Variational Analysis}~21 (2013), 431--473.

\bibitem{BMW2}
H.H.\ Bauschke, W.L.\ Hare, and W.M.\ Moursi,
Generalized solutions for the sum of two maximally monotone
operators,
\emph{SIAM Journal on Control and Optimization}, in press. 

\bibitem{BT}
J.M\ Borwein and M.K.\ Tam,
A cyclic Douglas--Rachford iteration scheme,
\emph{Journal of Optimization Theory and Applications}~160 (2014), 1--29. 

\bibitem{Ceg}
A.\ Cegielski,
\emph{Iterative Methods for Fixed Point Problems in Hilbert
Spaces},
Springer, 2012. 

\bibitem{CZ}
Y.\ Censor and S.A.\ Zenios,
\emph{Parallel Optimization},
Oxford University Press, 1997.

\bibitem{Comb96}
P.L.\ Combettes,
The convex feasibility problem in image recovery, 
\emph{Advances in Imaging and Electronic Physics}~95 (1996),
155--270. 

\bibitem{Comb04}
P.L.\ Combettes,
Solving monotone inclusions via compositions of nonexpansive
averaged operators,
\emph{Optimization}~53 (2004), 475--504. 


\bibitem{Deutsch}
F.\ Deutsch,
\emph{Best Approximation in Inner Product Spaces},
Springer, 2001.

\bibitem{DH}
F.\ Deutsch and H.\ Hundal,
The rate of convergence for the cyclic projections algorithm III:
regularity of convex sets,
\emph{Journal of Approximation Theory}~155 (2008), 155--184. 

\bibitem{DonRock}
A.L.\ Dontchev and R.T.\ Rockafellar,
\emph{Implicit Functions and Solution Mappings},
Springer 2009. 

\bibitem{DKLR}
J.M.\ Dye, T.\ Kuczumow, P.-K.\ Lin, and S.\ Reich,
Random products of nonexpansive mappings in spaces with the
Opial property, 
\emph{Contemporary Mathematics}~144 (1993), 87--93. 

\bibitem{GK}
K.\ Goebel and W.A.\ Kirk,
\emph{Topics in Metric Fixed Point Theory},
Cambridge University Press, 1990.

\bibitem{GR}
K.\ Goebel and S.\ Reich,
\emph{Uniform Convexity, Hyperbolic Geometry, and Nonexpansive
Mappings}, Marcel Dekker, 1984. 

\bibitem{Julia}
\texttt{http://julialang.org}

\bibitem{KL}
K.C.\ Kiwiel and B.\ {\L}opuch,
Surrogate projection methods for finding fixed points of
firmly nonexpansive mappings,
\emph{SIAM Journal on Optimiztion}~7 (1997), 1084--1102. 

\bibitem{Kru}
A.Y.\ Kruger,
About regularity of collections of sets,
\emph{Set-Valued Analysis}~14 (2006), 187--206.

\bibitem{Leventhal}
D.\ Leventhal, 
Metric subregularity and the proximal point method,
\emph{Journal of Mathematical Analysis and Applications}~360
(2009), 681--688. 

\bibitem{Hung}
H.M.\ Phan,
Linear convergence of the Douglas--Rachford method for two closed
sets, preprint 2014, \texttt{http://arxiv.org/abs/1401.6509}

\bibitem{Yamada}
I.\ Yamada,
The hybrid steepest descent method for the variational inequality
problem over the intersection of fixed point sets of nonexpansive
mappings,
\emph{Studies in Computational Mathematics}~8 (2001), 473--504. 


\end{thebibliography}
\end{document}